\documentclass[a4paper,10pt]{amsart}
\usepackage[colorlinks,linkcolor=blue,citecolor=blue]{hyperref}
\usepackage{latexsym, amssymb, amsmath, amsthm, mathrsfs, bbm}
\usepackage[all, knot]{xy}
\xyoption{arc}

\def \To{\longrightarrow}

\def \Hom{\operatorname{Hom}}
\def \Vec{\operatorname{Vec}}
\def \H{\operatorname{H}}
\def \T{\operatorname{T}}

\def \N{\mathbb{N}}

\def \R{\mathcal{R}}
\def \Z{\mathbb{Z}}

\numberwithin{equation}{section}

\newtheorem{theorem}{Theorem}[section]
\newtheorem{lemma}[theorem]{Lemma}
\newtheorem{proposition}[theorem]{Proposition}
\newtheorem{corollary}[theorem]{Corollary}
\newtheorem{definition}[theorem]{Definition}

\newtheorem{remark}[theorem]{Remark}

\begin{document}
\title[BRAIDED MONOIDAL GR-CATEGORIES]{THE BRAIDED MONOIDAL STRUCTURES ON A CLASS OF LINEAR GR-CATEGORIES}
 \thanks{Supported by the NSFC grants 10971206 and 11071111, the SDNSF grant 2009ZRA01128 and the IIFSDU grant 2010TS021.}
\author{Hua-Lin Huang}
\address{School of Mathematics, Shandong University, Jinan 250100, China} \email{hualin@sdu.edu.cn}
\author{Gongxiang Liu}
\address{Department of Mathematics, Nanjing University, Nanjing 210093, China}
\email{gxliu@nju.edu.cn}
\author{Yu Ye}
\address{School of Mathematics, USTC, Hefei 230026, China}
\email{yeyu@ustc.edu.cn}
\date{}
\maketitle

\begin{abstract}
A linear Gr-category is a category of finite-dimensional vector
spaces graded by a finite group together with the natural tensor product. We classify the braided monoidal
structures of a class of linear Gr-categories via explicit
computations of the normalized 3-cocycles and the quasi-bicharacters of finite
abelian groups which are direct product of two cyclic groups.
\end{abstract}

\section{Introduction}
By a linear Gr-category we mean a category of finite-dimensional vector spaces graded by a finite group together with the natural tensor product of graded vector spaces.  Linear Gr-categories arise naturally in such areas of mathematics as cohomology of groups, representation theory, tensor categories, and quantum groups. In 1975,
the monoidal structures of a Gr-category were related to the 3rd cohomology group of its grading group for the first time in the thesis of Ho\`{a}ng Xu\^{a}n S\'{i}nh
\cite{grc}. When the group is abelian, a Gr-category admits further a braiding and its braided monoidal structures are related to the 3rd abelian cohomology group, see \cite{js}.

We are mainly concerned about the applications of linear Gr-categories in the classification of finite pointed tensor categories \cite{egno, eo}. Note that, pointed fusion (i.e., semisimple tensor) categories are nothing other than linear Gr-categories \cite{eno}, moreover, the full subcategory of all semi-simple objects of a finite pointed tensor category is a linear Gr-category, therefore a thorough understanding of the monoidal structures of linear Gr-categories is indispensable for the purpose of our classification problem. In fact, the starting point of our investigation was the attempt to classify pointed tensor categories of tame type over the field of complex numbers. The full subcategory of semi-simple objects of any such tensor category is a linear Gr-category over a cyclic group or the direct product of two cyclic groups.

The crux for the classification of the monoidal structures on a linear Gr-category lies in an explicit and unified formula of the normalized 3-cocycles, not just the 3rd cohomology group, of the grading group. It is worthy to stress that, though the cohomology group of a finite group might be known, the explicit form of normalized cocycles is not necessarily clear. A naive reason is that, one may compute the cohomology group by the minimal (or any simpler) resolution, however one needs to work on the bar resolution to get normalized cocycles, see for example \cite{w}. In the case where the group is cyclic, the nice formula of the normalized 3-cocycles and the classification of the braided monoidal structures are presented in \cite{js}. These facts are important in the recent advances in the classification of finite pointed tensor categories whose invertible objects make cyclic groups, see for instance \cite{a,qha3}. However for non-cyclic groups, very few is known. To the best of our knowledge, the only result in this direction is \cite{bct} in which the group is the Klein four group and the computations therein are very technical and there seems no hope to extend to more general groups.

The aim of this note is to give explicit and unified formulae of the normalized 3-cocycles and the quasi-bicharacters for the direct product of two arbitrary finite cyclic groups, and hence provide a classification of the braided monoidal structures of the linear Gr-categories over such groups. Our main idea is to construct a chain map, up to the 3rd term, from the bar resolution to a simpler resolution. In the latter resolution, the cocycles are ready to be handled with. By the chain map, the computation of normalized 3-cocycles is thus transited to a much easier situation. Our results vastly extend those obtained in \cite{js,bct}. This also opens a door to the classification theory of finite pointed quasi-quantum groups and finite pointed tensor categories \cite{qha1,qha2,qha3} which will be a nontrivial generalization of the beautiful theory of finite pointed quantum groups, see \cite{as2} and references therein.

The note is organized as follows. In Section 2, we give an explicit chain map from the bar resolution to the minimal resolution of a finite cyclic group. This provides a hint to deal with the case of the direct product of two finite cyclic groups which is carried out in Section 3. The first three terms of a chain map are constructed and thus the normalized cocycles up to degree 3 and the quasi-bicharacters are obtained. In Section 4, we give the classification of the braided monoidal structures of linear Gr-categories over the direct product of any two finite cyclic groups.

Throughout, $k$ is an algebraically closed field of characteristic zero and let $k^{\ast}$ denote the multiplicative group $k\backslash\{0\}.$

\section{The cocycles of cyclic groups}
In this section, we construct a chain map from the bar resolution to the minimal resolution of any finite cyclic group. This leads to explicit and unified formulae for cocycles of all degrees. As mentioned above, the results up to 3-cocycles were known in literatures. The results for cocycles of higher degrees seem to be new.

Let $G=\mathbb{Z}_{m}=\langle g|g^{m}=1\rangle$ be the cyclic group of order $m.$ The trivial $\mathbb{Z}G$-module
$\mathbb{Z}$ has the following minimal resolution (see \cite[Section 6.2]{w})
\begin{equation}\cdots\longrightarrow \mathbb{Z}\mathbb{Z}_{m}\stackrel{g-1}\longrightarrow
\mathbb{Z}\mathbb{Z}_{m}\stackrel{N_{m}}\longrightarrow\mathbb{Z}\mathbb{Z}_{m}\stackrel{g-1}\longrightarrow
\mathbb{Z}\mathbb{Z}_{m}\stackrel{N_{m}}\longrightarrow
\mathbb{Z}\longrightarrow 0,\end{equation}
where $N_{m}=\sum_{i=0}^{m-1}g^{i}$. Let $(M_{\bullet},d_{\bullet})$ denote this resolution. To avoid confusion, denote the generator of the $i$-th free module in (2.1) by $\Psi_{i}$ for $i\geq 0.$ Thus the differential of (2.1) can be described in the following way
$$
d(\Psi_{i})=\left \{
\begin{array}{ll}  (g-1)\Psi_{i-1} &\;\;\;\; i\;\textrm{odd}
\\N_{m} \Psi_{i-1}&
\;\;\;\; i\;\textrm{even}
\end{array}.\right. $$
%To classify the braided monoidal structure over a linear Gr-category , a technique and difficult point is to give a chain map between the bar resolution
%and the minimal resolution. In the case of cyclic groups, this problem has an easy description.

By $(B_{\bullet},\partial_{\bullet})$ we denote the bar resolution of the trivial $\mathbb{Z}G$-module
$\mathbb{Z}$. That is, $$B_{n}=\bigoplus_{0\leq i_{1},\ldots,i_{n}\leq m-1}\mathbb{Z}G[g^{i_{1}},\cdots,g^{i_{n}}]$$
and
\begin{eqnarray*}&&\partial_{n}([g^{i_{1}},\cdots,g^{i_{n}}])\\
&=&g^{i_{1}}[g^{i_{2}},\cdots,g^{i_{n}}]+\sum_{j=1}^{n-1}(-1)^{j}[g^{i_{1}},\cdots,g^{i_{j}}g^{i_{j+1}},\cdots,g^{i_{n}}]
+(-1)^{n}[g^{i_{1}},\cdots,g^{i_{n-1}}].
\end{eqnarray*}

The main objective of this section is to give a chain map from the bar resolution to the minimal resolution. As preparation, we need to fix some notations and give a technical lemma. For a natural number $i,$ we denote by $i'$ the remainder of division of $i$ by $m.$ Given a rational number $x,$ let $[x]$ denote the integer part of $x,$ i.e., the largest integer not greater than $x.$ The following technical lemma is useful in our later arguments and will be used freely.

\begin{lemma} For any two natural numbers $i$ and $j,$ we have
\begin{equation}[\frac{i+j'}{m}]=[\frac{i+j}{m}]-[\frac{j}{m}].
\end{equation}\end{lemma}

\begin{proof} $ [\frac{i+j'}{m}]=[\frac{i+j-[\frac{j}{m}]m}{m}]=[\frac{i+j}{m}]-[\frac{j}{m}].$ \end{proof}

Now we are ready to give the desired chain map $F: (B_{\bullet},\partial_{\bullet}) \To (M_{\bullet},d_{\bullet}).$ Let
\begin{gather}F_{2k+1}:\;[g^{i_{1}},\cdots,g^{i_{2k+1}}]\mapsto \sum_{\alpha=0}^{i_{1}-1}[\frac{i_{2}+i_{3}}{m}]\cdots[\frac{i_{2k}+i_{2k+1}}{m}]g^{\alpha}\Psi_{2k+1},\;\;\;\;k\geq 0;\\
F_{2k}:\;[g^{i_{1}},\cdots,g^{i_{2k}}]\mapsto [\frac{i_{1}+i_{2}}{m}]\cdots[\frac{i_{2k-1}+i_{2k}}{m}]\Psi_{2k},\;\;\;\;k\geq 1.\notag
\end{gather} Here, if $i_{1}=0$, then $\sum_{\alpha=0}^{i_{1}-1}g^{\alpha}$ is understood as zero.

\begin{lemma} The map $F=\{F_{i}|i\geq 1\}$ defined above is a chain map from the bar resolution to the minimal resolution.
\end{lemma}
\begin{proof} We verify the claim by direct computations. Indeed, for any  $0\leq i_{1},\ldots,i_{2k+1}\leq m-1$, we have
\begin{eqnarray*}
&&F_{2k}(\partial_{2k+1}([g^{i_{1}},\cdots,g^{i_{2k+1}}]))\\
&=&F_{2k}(g^{i_{1}}[g^{i_{2}},\cdots,g^{i_{2k+1}}]
+\sum_{j=1}^{2k}(-1)^{j}[g^{i_{1}},\cdots,g^{i_{j}}g^{i_{j+1}},\cdots,g^{i_{2k+1}}]
-[g^{i_{1}},\cdots,g^{i_{2k}}])\\
&=&(g^{i_{1}}[\frac{i_{2}+i_{3}}{m}]\cdots[\frac{i_{2k}+i_{2k+1}}{m}]\\
&& -[\frac{(i_{1}+i_{2})'+i_{3}}{m}]\cdots [\frac{i_{2k}+i_{2k+1}}{m}]+[\frac{i_{1}+(i_{2}+i_{3})'}{m}]\cdots [\frac{i_{2k}+i_{2k+1}}{m}]
-\cdots\\
&&-[\frac{i_{1}+i_{2}}{m}]\cdots[\frac{(i_{2k-1}+i_{2k})'+i_{2k+1}}{m}]+[\frac{i_{1}+i_{2}}{m}]\cdots[\frac{i_{2k-1}+(i_{2k}+i_{2k+1})'}{m}]
\\
&&-[\frac{i_{1}+i_{2}}{m}]\cdots[\frac{i_{2k-1}+i_{2k}}{m}])\Psi_{2k}\\
&=&(g^{i_{1}}[\frac{i_{2}+i_{3}}{m}]\cdots[\frac{i_{2k}+i_{2k+1}}{m}]\\
&& +[\frac{i_{1}+i_{2}}{m}]\cdots [\frac{i_{2k}+i_{2k+1}}{m}]-[\frac{i_{2}+i_{3}}{m}]\cdots [\frac{i_{2k}+i_{2k+1}}{m}]
-\cdots\\
&&+[\frac{i_{1}+i_{2}}{m}]\cdots[\frac{i_{2k-1}+i_{2k}}{m}]+[\frac{i_{1}+i_{2}}{m}]\cdots[\frac{i_{2k}+i_{2k+1}}{m}]
\\
&&-[\frac{i_{1}+i_{2}}{m}]\cdots[\frac{i_{2k-1}+i_{2k}}{m}])\Psi_{2k}\\
&=&[\frac{i_{2}+i_{3}}{m}]\cdots[\frac{i_{2k}+i_{2k+1}}{m}](g^{i_{1}}-1)\Psi_{2k}\\
&=&\sum_{\alpha=0}^{i_{1}-1}[\frac{i_{2}+i_{3}}{m}]\cdots[\frac{i_{2k}+i_{2k+1}}{m}]g^{\alpha}(g-1)\Psi_{2k}\\
&=&dF_{2k+1}([g^{i_{1}},\cdots,g^{i_{2k+1}}]),
\end{eqnarray*}
and
\begin{eqnarray*}
&&F_{2k-1}(\partial_{2k}([g^{i_{1}},\cdots,g^{i_{2k}}]))\\
&=&F_{2k}(g^{i_{1}}[g^{i_{2}},\cdots,g^{i_{2k}}]
+\sum_{j=1}^{2k-1}(-1)^{j}[g^{i_{1}},\cdots,g^{i_{j}}g^{i_{j+1}},\cdots,g^{i_{2k}}]
+[g^{i_{1}},\cdots,g^{i_{2k-1}}])\\
&=&(g^{i_{1}}\sum_{\alpha=0}^{i_{2}-1}[\frac{i_{3}+i_{4}}{m}]\cdots[\frac{i_{2k-1}+i_{2k}}{m}]g^{\alpha}-
\sum_{\alpha=0}^{(i_{1}+i_{2})'-1}[\frac{i_{3}+i_{4}}{m}]\cdots[\frac{i_{2k-1}+i_{2k}}{m}]g^{\alpha}\\
&&+\sum_{\alpha=0}^{i_{1}-1}[\frac{(i_{2}+i_{3})'+i_{4}}{m}]\cdots[\frac{i_{2k-1}+i_{2k}}{m}]g^{\alpha} - \sum_{\alpha=0}^{i_{1}-1}[\frac{i_{2}+(i_{3}+i_{4})'}{m}]\cdots[\frac{i_{2k-1}+i_{2k}}{m}]g^{\alpha}\\
&&+\cdots
+\sum_{\alpha=0}^{i_{1}-1}[\frac{i_{2}+i_{3}}{m}]\cdots[\frac{(i_{2k-2}+i_{2k-1})'+i_{2k}}{m}]g^{\alpha}\\
&&-
\sum_{\alpha=0}^{i_{1}-1}[\frac{i_{2}+i_{3}}{m}]\cdots[\frac{i_{2k-2}+(i_{2k-1}+i_{2k})'}{m}]g^{\alpha}\\
&&+\sum_{\alpha=0}^{i_{1}-1}[\frac{i_{2}+i_{3}}{m}]\cdots[\frac{i_{2k-2}+i_{2k-1}}{m}]g^{\alpha})\Psi_{2k-1}\\
&=&(g^{i_{1}}\sum_{\alpha=0}^{i_{1}-1}[\frac{i_{3}+i_{4}}{m}]\cdots[\frac{i_{2k-1}+i_{2k}}{m}]g^{\alpha}-
\sum_{\alpha=0}^{(i_{1}+i_{2})-1}[\frac{i_{3}+i_{4}}{m}]\cdots[\frac{i_{2k-1}+i_{2k}}{m}]g^{\alpha}\\
&&+[\frac{i_{1}+i_{2}}{m}][\frac{i_{3}+i_{4}}{m}]\cdots[\frac{i_{2k-1}+i_{2k}}{m}]N_{m}\\
&&-\sum_{\alpha=0}^{i_{1}-1}[\frac{i_{2}+i_{3}}{m}]\cdots[\frac{i_{2k-1}+i_{2k}}{m}]g^{\alpha} + \sum_{\alpha=0}^{i_{1}-1}[\frac{i_{3}+i_{4}}{m}]\cdots[\frac{i_{2k-1}+i_{2k}}{m}]g^{\alpha} \\
&&+\cdots-\sum_{\alpha=0}^{i_{1}-1}[\frac{i_{2}+i_{3}}{m}]\cdots[\frac{i_{2k-2}+i_{2k-1}}{m}]g^{\alpha}
+\sum_{\alpha=0}^{i_{1}-1}[\frac{i_{2}+i_{3}}{m}]\cdots[\frac{i_{2k-1}+i_{2k}}{m}]g^{\alpha}\\
&&+\sum_{\alpha=0}^{i_{1}-1}[\frac{i_{2}+i_{3}}{m}]\cdots[\frac{i_{2k-2}+i_{2k-1}}{m}]g^{\alpha})\Psi_{2k-1}\\
&=&(g^{i_{1}}\sum_{\alpha=0}^{i_{1}-1}[\frac{i_{3}+i_{4}}{m}]\cdots[\frac{i_{2k-1}+i_{2k}}{m}]g^{\alpha}-
\sum_{\alpha=0}^{(i_{1}+i_{2})-1}[\frac{i_{3}+i_{4}}{m}]\cdots[\frac{i_{2k-1}+i_{2k}}{m}]g^{\alpha}\\
&&+[\frac{i_{1}+i_{2}}{m}][\frac{i_{3}+i_{4}}{m}]\cdots[\frac{i_{2k-1}+i_{2k}}{m}]N_{m}+\sum_{\alpha=0}^{i_{1}-1}[\frac{i_{3}+i_{4}}{m}]\cdots[\frac{i_{2k-1}+i_{2k}}{m}]g^{\alpha})\Psi_{2k-1}\\
&=&[\frac{i_{1}+i_{2}}{m}][\frac{i_{3}+i_{4}}{m}]\cdots[\frac{i_{2k-1}+i_{2k}}{m}]N_{m}\Psi_{2k-1}\\
&=&dF_{2k}([g^{i_{1}},\cdots,g^{i_{2k}}]).
\end{eqnarray*}
\end{proof}

Since $k$ is algebraically closed, $\H^{l}(\mathbb{Z}_{m},k^{\ast})\cong k^{\ast}/(k^{\ast})^{m}=0$ if $l$ is even. So there is no non-trivial $l$-cocycle whenever $l$ is even.

\begin{proposition} Suppose $l$ is odd and $\zeta_{m}$ an $m$-th primitive root of unity. Then the set of maps $$\omega_{a}:\;B_{l}\to k^{\ast},\;\;
[g^{i_{1}},\cdots,g^{i_{l}}]\mapsto \zeta_{m}^{ai_{1}[\frac{i_{2}+i_{3}}{m}]\cdots[\frac{i_{l-1}+i_{l}}{m}]},\;\;\;\;0\leq a< m$$
forms a complete set of representatives of $l$-cocycles.
\end{proposition}
\begin{proof} It is well known that $\H^{l}(\mathbb{Z}_{m},k^{\ast})\cong \mathbb{Z}_{m}$. So it is enough to show that $\omega_{a}$ is an $l$-cocycle. Consider the minimal resolution,
we know any $l$-cochain $f:(\mathbb{Z}G)\Psi_{l}\to k^{\ast}$ is uniquely determined by the value $f_{l}:=f(\Psi_{l})$. It is not hard to see that $f$ is a cocycle if and only if
$f_{l}$ is an $m$-th root of unity. So there is some natural number $a \in [0,m)$ such that $f_{l}=\zeta_{m}^{a}$. By Lemma 2.2, $\omega_{a}$ is an $l$-cocycle and they indeed form a complete set of representatives of $l$-cocycles.
\end{proof}

\section{The cocycles and quasi-bicharacters of $\mathbb{Z}_{m}\times \mathbb{Z}_{n}$}
In this section, firstly we give a $\mathbb{Z}_{m}\times \mathbb{Z}_{n}$ -resolution of $\mathbb{Z}$ which is the tensor product of the minimal resolutions of cyclic factors as given in Section 2, then provide the first 3 terms of a chain map from the bar resolution to this resolution. This enables us to obtain explicit formulae for the desired 3-cocycles and quasi-bicharacters. As byproducts, we also get some results on 2-cocycles and 2nd cohomology group which are obviously important in the cohomology and representation theory of groups.

\subsection{A resolution.}
Let $g_{1}$ (resp. $g_{2}$) be a generator of $\mathbb{Z}_{m}$
(resp. $\mathbb{Z}_{n}$). The norm in $\mathbb{Z}\mathbb{Z}_{m}$ is
the element $N_{m}=\sum_{i=0}^{m-1}g_{1}^{i}$. As given in Section 2, the
following periodic sequence is a projective resolution of the
trivial $\mathbb{Z}_{m}$-module $\mathbb{Z}$
\begin{equation}\cdots\longrightarrow \mathbb{Z}\mathbb{Z}_{m}\stackrel{g_{1}-1}\longrightarrow
\mathbb{Z}\mathbb{Z}_{m}\stackrel{N_{m}}\longrightarrow\mathbb{Z}\mathbb{Z}_{m}\stackrel{g_{1}-1}\longrightarrow
\mathbb{Z}\mathbb{Z}_{m}\stackrel{N_{m}}\longrightarrow
\mathbb{Z}\longrightarrow 0.\end{equation} Denote $g_{1}-1$ by $T_{m}$ for
convenience. Similarly one can define $N_{n}$ and $T_{n}$ and has the following projective resolution
of the trivial $\mathbb{Z}_{n}$-module $\mathbb{Z}$
\begin{equation}\cdots\longrightarrow \mathbb{Z}\mathbb{Z}_{n}\stackrel{T_{n}}\longrightarrow
\mathbb{Z}\mathbb{Z}_{n}\stackrel{N_{n}}\longrightarrow\mathbb{Z}\mathbb{Z}_{n}\stackrel{T_{n}}\longrightarrow
\mathbb{Z}\mathbb{Z}_{n}\stackrel{N_{n}}\longrightarrow
\mathbb{Z}\longrightarrow 0.\end{equation}

We construct the tensor product of above periodic resolutions for
$\mathbb{Z}_{m}$ and $\mathbb{Z}_{n}$. Let $K_{\bullet}$ be the following complex of projective (actually,
free) $\mathbb{Z}(\mathbb{Z}_{m}\times \mathbb{Z}_{n})$-modules. For
each pair $(i,j)$ of nonnegative integers, let $\Psi(i,j)$ be a free
generator in degree $i+j$. Thus
$$K_{l}:=\bigoplus_{i+j=l} \mathbb{Z}(\mathbb{Z}_{m}\times
\mathbb{Z}_{n})\Psi(i,j).$$ For the differential, define
$$d_{1}(\Psi(i,j))=\left \{
\begin{array}{lll} 0 &\;\;\;\;i=0
\\ N_{m}\Psi(i-1,j) & \;\;\;\;0\neq i\;\textrm{even}
\\T_{m} \Psi(i-1,j)&
\;\;\;\;0\neq i\;\textrm{odd}
\end{array};\right.$$
$$
d_{2}(\Psi(i,j))=\left \{
\begin{array}{lll}  0 &\;\;\;\;j=0\\(-1)^{i}N_{n}\Psi(i,j-1) & \;\;\;\;0\neq j\;\textrm{even}
\\(-1)^{i}T_{n} \Psi(i,j-1)&
\;\;\;\;0\neq j\;\textrm{odd}
\end{array}.\right. $$
The differential $d$ is just defined to be $d_{1}+d_{2}$.

\begin{lemma} $(K_{\bullet},d)$ is a free resolution of trivial $\mathbb{Z}(\mathbb{Z}_{m}\times
\mathbb{Z}_{n})$-module $\mathbb{Z}$.
\end{lemma}
\begin{proof} By observing that $(K_{\bullet},d)$ is exactly the tensor product complex of (3.1) and
(3.2), the lemma follows by the K\"unneth formula for complexes (see (3.6.3) in \cite{w}).
\end{proof}

\subsection{A chain map.} Let $B_{\bullet}\to \mathbb{Z}$ be the bar resolution of
the trivial $\mathbb{Z}(\mathbb{Z}_{m}\times \mathbb{Z}_{n})$-module
$\mathbb{Z}$ (see \cite[Section 6.5]{w} and notations therein). Thus
up to homotopy we have a unique chain map
$F_{\bullet}:B_{\bullet}\To K_{\bullet}$. For our purpose,
we write $F_{1}, F_{2}$ and $F_{3}$ out. Explicitly, for $0\leq
i,s,k< m, \ 0\leq j,t,l< n$,
\begin{gather}
F_{1}: [g_{1}^{i}g_{2}^{j}]\mapsto
\sum_{\alpha=0}^{i-1}g_{1}^{\alpha}\Psi(1,0)+
\sum_{\beta=0}^{j-1}g_{1}^{i}g_{2}^{\beta}\Psi(0,1);\\
F_{2}: [g_{1}^{i}g_{2}^{j},g_{1}^{s}g_{2}^{t}]\mapsto
[\frac{i+s}{m}]\Psi(2,0)-
\sum_{\alpha=0}^{s-1}\sum_{\beta=0}^{j-1}g_{1}^{\alpha+i}g_{2}^{\beta}\Psi(1,1)\\
+[\frac{j+t}{n}]g_{1}^{i+s}\Psi(0,2);\notag\\
F_{3} :
[g_{1}^{i}g_{2}^{j},g_{1}^{s}g_{2}^{t},g_{1}^{k}g_{2}^{l}]\mapsto
\sum_{\alpha=0}^{i-1}[\frac{k+s}{m}]
g_{1}^{\alpha}\Psi(3,0)+\sum_{\beta=0}^{j-1}[\frac{k+s}{m}]g_{1}^{i}g_{2}^{\beta}\Psi(2,1)\\
+\sum_{\alpha=0}^{k-1}[\frac{j+t}{n}]g_{1}^{i+s+\alpha}\Psi(1,2)+
\sum_{\beta=0}^{j-1}[\frac{t+l}{n}]g_{1}^{i+s+k}g_{2}^{\beta}\Psi(0,3)\notag.
\end{gather}
Here, say, if $i=0$, we
understand $\sum_{\alpha=0}^{i-1}g_{1}^{i}$ as $0$.
\begin{lemma} The following diagram is commutative

\begin{figure}[hbt]
\begin{picture}(150,50)(50,-40)
\put(0,0){\makebox(0,0){$ \cdots$}}\put(10,0){\vector(1,0){20}}\put(40,0){\makebox(0,0){$B_{3}$}}
\put(50,0){\vector(1,0){20}}\put(80,0){\makebox(0,0){$B_{2}$}}
\put(90,0){\vector(1,0){20}}\put(120,0){\makebox(0,0){$B_{1}$}}
\put(130,0){\vector(1,0){20}}\put(160,0){\makebox(0,0){$B_{0}$}}
\put(170,0){\vector(1,0){20}}\put(200,0){\makebox(0,0){$\mathbb{Z}$}}
\put(210,0){\vector(1,0){20}}\put(240,0){\makebox(0,0){$0$}}

\put(0,-40){\makebox(0,0){$ \cdots$}}\put(10,-40){\vector(1,0){20}}\put(40,-40){\makebox(0,0){$K_{3}$}}
\put(50,-40){\vector(1,0){20}}\put(80,-40){\makebox(0,0){$K_{2}$}}
\put(90,-40){\vector(1,0){20}}\put(120,-40){\makebox(0,0){$K_{1}$}}
\put(130,-40){\vector(1,0){20}}\put(160,-40){\makebox(0,0){$K_{0}$}}
\put(170,-40){\vector(1,0){20}}\put(200,-40){\makebox(0,0){$\mathbb{Z}$}}
\put(210,-40){\vector(1,0){20}}\put(240,-40){\makebox(0,0){$0$}}

\put(40,-10){\vector(0,-1){20}}
\put(80,-10){\vector(0,-1){20}}
\put(120,-10){\vector(0,-1){20}}
\put(158,-10){\line(0,-1){20}}\put(160,-10){\line(0,-1){20}}
\put(198,-10){\line(0,-1){20}}\put(200,-10){\line(0,-1){20}}

\put(60,5){\makebox(0,0){$\partial_{3}$}}
\put(100,5){\makebox(0,0){$\partial_{2}$}}
\put(140,5){\makebox(0,0){$\partial_{1}$}}

\put(60,-35){\makebox(0,0){$d$}}
\put(100,-35){\makebox(0,0){$d$}}
\put(140,-35){\makebox(0,0){$d$}}

\put(50,-20){\makebox(0,0){$F_{3}$}}
\put(90,-20){\makebox(0,0){$F_{2}$}}
\put(130,-20){\makebox(0,0){$F_{1}$}}

\end{picture}
\end{figure}
\end{lemma}

\begin{proof} We adopt the notations used in \cite{w}. Take the generator $[g_{1}^{i}g_{2}^{j}]\in B_{1}$ for some $0\leq i<m$ and
$0\leq j< n$, then \[\partial_{1}([g_{1}^{i}g_{2}^{j}])=(g_{1}^{i}g_{2}^{j}-1)\Psi(0,0).\] On the other hand, \begin{eqnarray*}
dF_{1}([g_{1}^{i}g_{2}^{j}])&=&d(\sum_{\alpha=0}^{i-1}g_{1}^{\alpha}\Psi(1,0)+
\sum_{\beta=0}^{j-1}g_{1}^{i}g_{2}^{\beta}\Psi(0,1))\\ &=& (\sum_{\alpha=0}^{i-1}g_{1}^{\alpha}(g_{1}-1)+
\sum_{\beta=0}^{j-1}g_{1}^{i}g_{2}^{\beta}(g_{2}-1))\Psi(0,0)\\ &=&(g_{1}^{i}g_{2}^{j}-1)\Psi(0,0). \end{eqnarray*} So we have $dF_{1}=\partial_{1}$.

For any natural number $i,$ we denote by $i'$ and $i''$ the remainders of division of $i$ by $m$
and $n$ respectively. Now for any generator $[g_{1}^{i}g_{2}^{j},g_{1}^{s}g_{2}^{t}]\in B_{2}$,
\begin{eqnarray*}
&&F_{1}\partial_{2}([g_{1}^{i}g_{2}^{j},g_{1}^{s}g_{2}^{t}])\\
&=&F_{1}(g_{1}^{i}g_{2}^{j}[g_{1}^{s}g_{2}^{t}]-[g_{1}^{i+s}g_{2}^{j+t}]+[g_{1}^{i}g_{2}^{j}])\\
&=&g_{1}^{i}g_{2}^{j}(\sum_{\alpha=0}^{s-1}g_{1}^{\alpha}\Psi(1,0)+
\sum_{\beta=0}^{t-1}g_{1}^{s}g_{2}^{\beta}\Psi(0,1))
-(\sum_{\alpha=0}^{(i+s)'-1}g_{1}^{\alpha}\Psi(1,0)\\&&+
\sum_{\beta=0}^{(j+t)''-1}g_{1}^{i+s}g_{2}^{\beta}\Psi(0,1))
+(\sum_{\alpha=0}^{i-1}g_{1}^{\alpha}\Psi(1,0)+
\sum_{\beta=0}^{j-1}g_{1}^{i}g_{2}^{\beta}\Psi(0,1))\\
&=&g_{1}^{i}g_{2}^{j}(\sum_{\alpha=0}^{s-1}g_{1}^{\alpha}\Psi(1,0)+
\sum_{\beta=0}^{t-1}g_{1}^{s}g_{2}^{\beta}\Psi(0,1))
-((\sum_{\alpha=0}^{(i+s)-1}g_{1}^{\alpha}-[\frac{i+s}{m}]N_{m})\Psi(1,0)\\&&+
(\sum_{\beta=0}^{(j+t)-1}g_{1}^{i+s}g_{2}^{\beta}-g_{1}^{i+s}[\frac{j+t}{n}]N_{n})\Psi(0,1))
+(\sum_{\alpha=0}^{i-1}g_{1}^{\alpha}\Psi(1,0)+
\sum_{\beta=0}^{j-1}g_{1}^{i}g_{2}^{\beta}\Psi(0,1))\\
&=&([\frac{i+s}{m}]N_{m}+g_{1}^{i}(g_{2}^{j}-1)\sum_{\alpha=0}^{s-1}g_{1}^{\alpha})\Psi(1,0)\\
&&+(g_{1}^{i+s}[\frac{j+t}{n}]N_{n}+g_{1}^{i}(1-g_{1}^{s})\sum_{\beta=1}^{j-1}g_{2}^{\beta})\Psi(0,1)
\end{eqnarray*}
and
\begin{eqnarray*}
&&dF_{2}([g_{1}^{i}g_{2}^{j},g_{1}^{s}g_{2}^{t}])\\
&=&d([\frac{i+s}{m}]\Psi(2,0)-
\sum_{\alpha=0}^{s-1}\sum_{\beta=0}^{j-1}g_{1}^{\alpha+i}g_{2}^{\beta}\Psi(1,1)
+[\frac{j+t}{n}]g_{1}^{i+s}\Psi(0,2))\\
&=& [\frac{i+s}{m}]N_{m}\Psi(1,0) -\sum_{\alpha=0}^{s-1}\sum_{\beta=0}^{j-1}g_{1}^{\alpha+i}(g_{1}-1)g_{2}^{\beta}\Psi(0,1)\\
&&+\sum_{\alpha=0}^{s-1}\sum_{\beta=0}^{j-1}g_{1}^{\alpha+i}g_{2}^{\beta}(g_{2}-1)\Psi(1,0)+g_{1}^{i+s}[\frac{j+t}{n}]N_{n}\Psi(0,1)\\
&=&([\frac{i+s}{m}]N_{m}+g_{1}^{i}(g_{2}^{j}-1)\sum_{\alpha=0}^{s-1}g_{1}^{\alpha})\Psi(1,0)\\
&&+(g_{1}^{i+s}[\frac{j+t}{n}]N_{n}+g_{1}^{i}(1-g_{1}^{s})\sum_{\beta=1}^{j-1}g_{2}^{\beta})\Psi(0,1).
\end{eqnarray*}
So, we have proved that $F_{1}\partial_{2}=dF_{2}$. At last,
\begin{eqnarray*}
&&F_{2}\partial_{3}([g_{1}^{i}g_{2}^{j},g_{1}^{s}g_{2}^{t},g_{1}^{k}g_{2}^{l}])\\
&=&F_{2}(g_{1}^{i}g_{2}^{j}[g_{1}^{s}g_{2}^{t},g_{1}^{k}g_{2}^{l}]-[g_{1}^{i+s}g_{2}^{j+t},g_{1}^{k}g_{2}^{l}]+[g_{1}^{i}g_{2}^{j},g_{1}^{s+k}g_{2}^{t+l}]
-[g_{1}^{i}g_{2}^{j},g_{1}^{s}g_{2}^{t}])\\
&=&g_{1}^{i}g_{2}^{j}([\frac{s+k}{m}]\Psi(2,0)-\sum_{\alpha=0}^{k-1}\sum_{\beta=0}^{t-1}g_{1}^{\alpha+s}g_{2}^{\beta}\Psi(1,1)+[\frac{t+l}{n}]g_{1}^{s+k}\Psi(0,2))\\
&-&([\frac{(i+s)'+k}{m}]\Psi(2,0)-\sum_{\alpha=0}^{k-1}\sum_{\beta=0}^{(j+t)''-1}g_{1}^{\alpha+i+s}g_{2}^{\beta}\Psi(1,1)+[\frac{(j+t)''+l}{n}]g_{1}^{i+s+k}\Psi(0,2))\\
&+&([\frac{i+(s+k)'}{m}]\Psi(2,0)-\sum_{\alpha=0}^{(s+k)'-1}\sum_{\beta=0}^{j-1}g_{1}^{\alpha+i}g_{2}^{\beta}\Psi(1,1)+[\frac{j+(t+l)''}{n}]g_{1}^{i+s+k}\Psi(0,2))\\
&-&([\frac{i+s}{m}]\Psi(2,0)-\sum_{\alpha=0}^{s-1}\sum_{\beta=0}^{j''-1}g_{1}^{\alpha+i}g_{2}^{\beta}\Psi(1,1)+[\frac{j+t}{n}]g_{1}^{i+s+k}\Psi(0,2)).
\end{eqnarray*}
In such expression, the item containing $\Psi(2,0)$ equals to
\begin{eqnarray*}
&&([\frac{s+k}{m}]g_{1}^{i}g_{2}^{j}-([\frac{i+s+k}{m}]-[\frac{i+s}{m}])+([\frac{i+s+k}{m}]-[\frac{s+k}{m}])-[\frac{i+s}{m}])\Psi(2,0)\\
&&=([\frac{s+k}{m}]g_{1}^{i}g_{2}^{j}-[\frac{s+k}{m}])\Psi(2,0)\\
&&=[\frac{s+k}{m}](g_{1}^{i}g_{2}^{j}-1)\Psi(2,0).
\end{eqnarray*}
The item containing $\Psi(1,1)$ equals to
\begin{eqnarray*}
&&(-\sum_{\alpha=0}^{k-1}\sum_{\beta=j}^{j+t-1}g_{1}^{\alpha+i+s}g_{2}^{\beta}+\sum_{\alpha=0}^{k-1}g_{1}^{\alpha+i+s}(\sum_{\beta=0}^{j+t-1}g_{1}^{\alpha+s}g_{2}^{\beta}-[\frac{j+t}{n}]N_{n})\\
&&-\sum_{\alpha=0}^{(s+k)-1}(g_{1}^{\alpha+i}-[\frac{s+k}{m}]g_{1}^{i}N_{m})\sum_{\beta=0}^{j-1}g_{2}^{\beta}+\sum_{\alpha=0}^{s-1}g_{1}^{\alpha+i}\sum_{\beta=0}^{j''-1}g_{2}^{\beta})\Psi(1,1)\\
&=&(\sum_{\alpha=0}^{k-1}g_{1}^{\alpha+i+s}\sum_{\beta=0}^{j-1}g_{2}^{\beta}-\sum_{\alpha=0}^{k-1}g_{1}^{\alpha+i+s}[\frac{j+t}{n}]N_{n}\\
&&+\sum_{\beta=0}^{j-1}g_{2}^{\beta}[\frac{s+k}{m}]g_{1}^{i}N_{m}-\sum_{\alpha=0}^{k-1}g_{1}^{\alpha+i+s}\sum_{\beta=0}^{j-1}g_{2}^{\beta})\Psi(1,1)\\
&=&(\sum_{\beta=0}^{j-1}g_{2}^{\beta}[\frac{s+k}{m}]g_{1}^{i}N_{m}-\sum_{\alpha=0}^{k-1}g_{1}^{\alpha+i+s}[\frac{j+t}{n}]N_{n})\Psi(1,1)
\end{eqnarray*}
The item containing $\Psi(0,2)$ equals to
\begin{eqnarray*}
&&([\frac{t+l}{n}]g_{1}^{i+s+k}g_{2}^{j}-([\frac{j+t+l}{n}]-[\frac{j+t}{n}])g_{1}^{i+s+k}\\
&&+([\frac{j+t+l}{n}]-[\frac{t+l}{n}])g_{1}^{i+s+k}-[\frac{j+t}{n}]g_{1}^{i+s})\Psi(0,2)\\
&=&([\frac{t+l}{n}]g_{1}^{i+s+k}(g_{2}^{j}-1)+[\frac{j+t}{n}](g_{1}^{i+s+k}-g_{1}^{i+s}))\Psi(0,2).
\end{eqnarray*}
In addition,
\begin{eqnarray*}
&&dF_{3}([g_{1}^{i}g_{2}^{j},g_{1}^{s}g_{2}^{t},g_{1}^{k}g_{2}^{l}])\\
&=&d(\sum_{\alpha=0}^{i-1}[\frac{k+s}{m}]g_{1}^{\alpha}\Psi(3,0)+\sum_{\beta=0}^{j-1}[\frac{k+s}{m}]g_{1}^{i}g_{2}^{\beta}\Psi(2,1)\\
&&+\sum_{\alpha=0}^{k-1}[\frac{j+t}{n}]g_{1}^{i+s+\alpha}\Psi(1,2)+
\sum_{\beta=0}^{j-1}[\frac{t+l}{n}]g_{1}^{i+s+k}g_{2}^{\beta}\Psi(0,3))\\
&=&\sum_{\alpha=0}^{i-1}[\frac{k+s}{m}]g_{1}^{\alpha}(g_{1}-1)\Psi(2,0)+\sum_{\beta=0}^{j-1}[\frac{k+s}{m}]g_{1}^{i}g_{2}^{\beta}(N_{m}\Psi(1,1)+(g_{2}-1)\Psi(2,0))\\
&&+\sum_{\alpha=0}^{k-1}[\frac{j+t}{n}]g_{1}^{i+s+\alpha}((g_{1}-1)\Psi(0,2)-N_{n}\Psi(1,1))+\sum_{\beta=0}^{j-1}[\frac{t+l}{n}]g_{1}^{i+s+k}g_{2}^{\beta}(g_{2}-1)\Psi(0,2)\\
&=&[\frac{k+s}{m}](g_{1}^{i}-1)\Psi(2,0)+[\frac{k+s}{m}]g_{1}^{i}(g_{2}^{j}-1)\Psi(2,0)\\
&&+[\frac{j+t}{n}](g_{1}^{i+s+k}-g_{1}^{i+s}\Psi(0,2)+[\frac{t+l}{n}]g_{1}^{i+s+k}(g_{2}^{j}-1)\Psi(0,2)\\
&&+\sum_{\beta=0}^{j-1}g_{2}^{\beta}[\frac{s+k}{m}]g_{1}^{i}N_{m}\Psi(1,1)-\sum_{\alpha=0}^{k-1}g_{1}^{\alpha+i+s}[\frac{j+t}{n}]N_{n}\Psi(1,1)\\
&=&[\frac{s+k}{m}](g_{1}^{i}g_{2}^{j}-1)\Psi(2,0)+(\sum_{\beta=0}^{j-1}g_{2}^{\beta}[\frac{s+k}{m}]g_{1}^{i}N_{m}-\sum_{\alpha=0}^{k-1}g_{1}^{\alpha+i+s}[\frac{j+t}{n}]N_{n})\Psi(1,1)\\
&&+([\frac{t+l}{n}]g_{1}^{i+s+k}(g_{2}^{j}-1)+[\frac{j+t}{n}](g_{1}^{i+s+k}-g_{1}^{i+s}))\Psi(0,2).
\end{eqnarray*}
By comparing the items containing $\Psi(2,0),\Psi(1,1)$ and $\Psi(0,2)$, we can find that $F_{2}\partial_{3}=dF_{3}$
\end{proof}

\subsection{2-cocycles and 2nd cohomology group.}
It is easy to see that a 2-cochain $f \in \Hom_{\mathbb{Z}(\mathbb{Z}_{m}\times \mathbb{Z}_{n})}(\bigoplus_{i+j=2} \mathbb{Z}(\mathbb{Z}_{m} \times \mathbb{Z}_{n}) \Psi(i,j),k^{\ast})$ is uniquely determined by the the sequence of  values
$(f(\Psi(2,0)),f(\Psi(1,1)),f(\Psi(0,2)))$.  For short,
let \[ A=f(\Psi(2,0)), \quad B=f(\Psi(1,1)), \quad \mathrm{and} \quad C=f(\Psi(0,2)).\]

\begin{lemma} The cochain $f$ is a $2$-cocycle if and only if $B^{m}=B^{n}=1$, and it is a $2$-coboundary if and only if
$B=1$.
\end{lemma}
\begin{proof} By definition, $f$ is a cocycle if and only if $d^{\ast}(f)=0$, which is equivalent to saying that
\begin{gather*}1=d^{\ast}(f)(\Psi(3,0))=f(d(\Psi(3,0)))=f(T_{m}\Psi(2,0))=A^{0},\\
1=d^{\ast}(f)(\Psi(2,1))=f(d(\Psi(2,1)))=f(N_{m}\Psi(1,1)+T_{n}\Psi(2,0))=B^{m}A^{0},\\
1=d^{\ast}(f)(\Psi(1,2))=f(d(\Psi(1,2)))=f(T_{m}\Psi(0,2)-N_{n}\Psi(1,1))=C^{0}B^{-n},\\
1=d^{\ast}(f)(\Psi(0,3))=f(d(\Psi(0,3)))=f(T_{n}\Psi(0,2))=C^{0}.
\end{gather*}
So the first part of the lemma is proved. For the second part, assume $f=d^{\ast}(g)$ for a 1-cochain $g$. Therefore,
\begin{gather*}A=d^{\ast}(g)(\Psi(2,0))=g(d(\Psi(2,0)))=g(N_{m}\Psi(1,0))=(g(\Psi(1,0)))^{m},\\
B=d^{\ast}(g)(\Psi(1,1))=g(d(\Psi(2,1)))=g(T_{m}\Psi(0,1)-T_{n}\Psi(1,0))\\
=(g(\Psi(0,1)))^{0}(g(\Psi(1,0)))^{0}=1,\\
C=d^{\ast}(g)(\Psi(0,2))=g(d(\Psi(0,2)))=g(N_{n}\Psi(0,1))=(g(\Psi(0,1)))^{n}.
\end{gather*}
Since the field is algebraically closed, the only restriction is  $B=1$.
\end{proof}
Let $(m,n)$ denote the greatest common divisor of $m$ and $n$. The preceding lemma implies the following result directly.
\begin{corollary} $\H^{2}(\mathbb{Z}_{m}\times
\mathbb{Z}_{n}, k^{\ast})\cong \mathbb{Z}_{(m,n)}$.
\end{corollary}

\begin{remark} For any natural number $s$, set $\zeta_{s}$ to be a primitive $s$-th root of unity. Thus the above
corollary tells us that $$\{f \;\mathrm{is\; a}\; 2\mathrm{-cochain}|f(\Psi(2,0))=1,f(\Psi(1,1))=\zeta_{(m,n)}^{b}, f(\Psi(0,2))=1,\; \mathrm{for}\;0
\leq b<(m,n)\}$$ is a complete set of representatives of
$2$-cocycles.
\end{remark}

By combining Lemmas 3.2 and 3.3, any $2$-cochain $\Phi$ in $ \Hom_{\Z(\Z_{m} \times \Z_{n})}({B}_{2},k^{\ast})$ can be described as
\begin{equation}\Phi(g_{1}^{i}g_{2}^{j},g_{1}^{s}g_{2}^{t})=A^{[\frac{i+s}{m}]}
B^{-js}C^{[\frac{j+t}{n}]}
\end{equation}
and it is a $2$-cocycle if and only if $B^{(m,n)}=1$.
Define $\Phi_{b}\in \Hom_{\Z(\Z_{m}\times
\Z_{n})}({B}_{2},k^{\ast})$ by
\begin{equation}
\Phi_{b}(g_{1}^{i}g_{2}^{j},g_{1}^{s}g_{2}^{t})
:=\zeta_{(m,n)}^{bjs}.\end{equation}

Owing to Remark 3.5, one has
\begin{proposition} The set $\{\Phi_{b}|0\leq b<(m,n)\}$ is a complete set of representatives of the normalized $2$-cocycles.
\end{proposition}

\subsection{3-cocycles and 3rd cohomology group.}
In the same manner as Subsection 3.3, any 3-cochain $f\in \Hom_{\mathbb{Z}(\mathbb{Z}_{m}\times
\mathbb{Z}_{n})}(\bigoplus_{i+j=3} \mathbb{Z}(\mathbb{Z}_{m}\times
\mathbb{Z}_{n})\Psi(i,j),k^{\ast})$ is uniquely
determined by its values on $\Psi(3,0),\Psi(2,1),\Psi(1,2)$ and $\Psi(0,3)$. By abuse of notation, let $A=f(\Psi(3,0)), B=f(\Psi(2,1)), C=f(\Psi(1,2))$ and $D=f(\Psi(0,3))$. Due to the similarities, we state the following results without proofs.

\begin{lemma} The map $f$ is a $3$-cocycle if and only if
$A^{m}=B^{n}C^{m}=D^{n}=1$, and it is a $3$-coboundary if and only if
$A=D=1, B=E^{m}$ and $C=E^{-n}$ for some $E\in k^{\ast}$.
\end{lemma}

\begin{corollary} $\H^{3}(\mathbb{Z}_{m}\times
\mathbb{Z}_{n}, k^{\ast})\cong \mathbb{Z}_{m}\times
\mathbb{Z}_{(m,n)}\times \mathbb{Z}_{n}$.
\end{corollary}

Using the third term $F_{3}$ of the chain map, any 3-cochain $\Phi$ in $ \Hom_{\Z(\Z_{m}\times \Z_{n})}({B}_{3},k^{\ast})$ can be described as
\begin{equation}\Phi(g_{1}^{i}g_{2}^{j},g_{1}^{s}g_{2}^{t},g_{1}^{k}g_{2}^{l})=A^{[\frac{k+s}{m}]i}
B^{[\frac{k+s}{m}]j}C^{[\frac{j+t}{n}]k}D^{[\frac{t+l}{n}]j}
\end{equation}
and it is a 3-cocycle if and only if $A^{m}=B^{n}C^{m}=D^{n}=1$.
Define a 3-cocycle $\Phi_{a,b,d}\in \Hom_{\Z(\Z_{m}\times \Z_{n})}({B}_{3},k^{\ast})$ by setting
\begin{equation}
\Phi_{a,b,d}(g_{1}^{i}g_{2}^{j},g_{1}^{s}g_{2}^{t},g_{1}^{k}g_{2}^{l})
=\zeta_{m}^{a[\frac{k+s}{m}]i}
\zeta_{n}^{b[\frac{k+s}{m}]j}\zeta_{n}^{d[\frac{t+l}{n}]j}.\end{equation}
\begin{proposition} The set $\{\Phi_{a,b,d}|0\leq a<m, 0\leq b<(m,n),0\leq d<n\}$ is a complete set of representatives of the normalized $3$-cocycles.
\end{proposition}

\subsection{Quasi-bicharacters.}
\begin{definition} Let $G$ be a group and $\Phi$ a normalized $3$-cocycle on $G$ with coefficients in $k^{\ast}$. A
map $\R: G \times G \To k^{\ast}$ is called a \emph{quasi-bicharacter with respect to $\Phi$}  provided the following equations are satisfied:
\begin{gather}
\R(xy, z)=\Phi(z , x , y)\R(x , z)\Phi^{-1}(x , z ,
y) \\ \ \ \times \R(y ,z)\Phi(x , y , z), \nonumber \\
\R(x , yz)=\Phi^{-1}(y , z , x)\R(x , z) \Phi(y , x ,
z) \\ \ \ \ \ \ \times \R(x
, y)\Phi^{-1}(x , y , z), \nonumber
\end{gather} for all $x,y,z \in G.$
\end{definition}

Clearly, if $\Phi$ is trivial, then a quasi-bicharacter is just an ordinary bicharacter. For the convenience of our later computations, we rewrite equations (3.10) and (3.11) in the following way:
\begin{gather}
\R(xy, z)=\R(x , z)\R(y ,
z)\frac{\Phi(z,x,y)\Phi(x,y,z)}{\Phi(x,z,y)},\\
\R(x , yz)=\R(x , y)\R(x , z)\frac{\Phi(y,x,z)}{\Phi(y,z,x)\Phi(x,y,z)},
\end{gather} for all $x,y,z \in G.$ The aim of this subsection is to describe all quasi-bicharacters of $\Z_{m}\times
\Z_{n}$. By Proposition 3.9, one can assume that $\Phi=\Phi_{a,b,d}$ for some $a\in \{0,\ldots,m-1\},
b\in \{0,\ldots,(m,n)-1\}$ and $d\in \{0,\ldots,n-1\}$. Clearly, any quasi-bicharacter $\R$ is uniquely determined by the following
four values:
$$r_{11}:=\R(g_{1},g_{1}),\;\;r_{12}:=\R(g_{1},g_{2}),\;\;r_{21}:=\R(g_{2},g_{1}),\;\;r_{22}:=\R(g_{2},g_{2}).$$

\begin{proposition} Let $G=\Z_m \times \Z_n, \ \Phi=\Phi_{a,b,d}$ and $r_{11},r_{12},r_{21},r_{22} \in k^*.$ Then there is
a quasi-bicharacter $\R$ with respect to $\Phi$ satisfying $\R(g_{1},g_{1})=r_{11},\R(g_{1},g_{2})=r_{12},\R(g_{2},g_{1})=r_{21},\R(g_{2},g_{2})=r_{22}$
if and only if the following equations are satisfied:
\begin{gather}
r_{11}^{m}=\zeta_{m}^{a}=\zeta_{m}^{-a},\;\; r_{22}^{n}=\zeta_{n}^{d}=\zeta_{n}^{-d}, \nonumber\\
r_{12}^{n}=1,\;\;r_{12}^{m}=\zeta_{n}^{-b},  \\
r_{21}^{n}=1,\;\;r_{21}^{m}=\zeta_{n}^{b}. \nonumber
\end{gather}
\end{proposition}
\begin{proof} By the definition of $\Phi_{a,b,d}$, it is easy to see that $\Phi_{a,b,d}(x,y,z)=\Phi_{a,b,d}(x,z,y)$. Then (3.12) and (3.13) can be
simplified into \begin{gather}
\R(xy, z)=\R(x , z)\R(y ,
z)\Phi(z,x,y),\\
\R(x , yz)=\R(x , y)\R(x , z)\frac{1}{\Phi(x,y,z)}
\end{gather} for all $x,y,z \in \Z_{m}\times
\Z_{n}.$

 ``$\Rightarrow$".  Using (3.15) and (3.16) iteratively, we have $\R(g_{1},g_{1}^{i})=\R(g_{1},g_{1})^{i}$ and $\R(g_{1}^{i},g_{1})=\R(g_{1},g_{1})^{i}$
for $1\leq i\leq m-1$. Now
$$1=\R(g_{1},g_{1}^{m})=\R(g_{1},g_{1})\R(g_{1},g_{1}^{m-1})\frac{1}{\Phi(g_{1},g_{1},g_{1}^{m-1})}=\R(g_{1},g_{1})^{m}\frac{1}{\zeta_{m}^{a}},$$
$$1=\R(g_{1}^{m},g_{1})=\R(g^{m-1}_{1},g_{1})\R(g_{1},g_{1})\Phi(g_{1},g^{m-1}_{1},g_{1})=\R(g_{1},g_{1})^{m}{\zeta_{m}^{a}}.$$
Thus $r_{11}^{m}=\zeta_{m}^{a}=\zeta_{m}^{-a}$.  Similarly, we can show that $r_{22}^{n}=\zeta_{n}^{d}=\zeta_{n}^{-d}$.

Again by applying (3.15) iteratively, one can show that $\R(g^{i}_{1},g_{2})=\R(g_{1},g_{2})^{i}$ for $1\leq i\leq m-1$. Therefore,
$$1=\R(g_{1}^{m},g_{2})=\R(g^{m-1}_{1},g_{2})\R(g_{1},g_{2})\Phi(g_{2},g^{m-1}_{1},g_{1})=\R(g_{1},g_{2})^{m}{\zeta_{n}^{b}}.$$
This implies that $r_{12}^{m}=\zeta_{n}^{-b}$. By the definition of $\Phi$, it is not hard to see that $\Phi(g_{1},g_{2}^{i},g_{2}^{j})\equiv 1$ for all $i,j$. Combining this fact and (3.16), we have $\R(g_{1},g^{i}_{2})=\R(g_{1},g_{2})^{i}$ for $1\leq i\leq n$. So
$$1=\R(g_{1},g^{n}_{2})=\R(g_{1},g_{2})^{n}=r_{12}^{n}.$$  Similarly, we have $r_{21}^{n}=1$ and $r_{21}^{m}=\zeta_{n}^{b}$. The necessity is proved.

``$\Leftarrow$". Conversely, define a map $\R: G \times G \To k^*$  by setting $$\R(g_{1}^{i}g_{2}^{j},g_{1}^{s}g_{2}^{t})=r_{11}^{is}r_{12}^{it}r_{21}^{js}r_{22}^{jt}.$$ To show it is indeed a quasi-bicharacter, it is
enough to prove that (3.15) and (3.16) are satisfied. We only show (3.15) since (3.16) can be proved similarly.  Recall that for any integer $i\in \mathbb{N}$,
we denote by $i'$ and $i''$ the remainders of division of $i$ by $m$
and $n$ respectively. Now,
$$\R(g_{1}^{s}g_{2}^{t}\cdot g_{1}^{k}g_{2}^{l},g_{1}^{i}g_{2}^{j})=\R(g_{1}^{(s+k)'}g_{2}^{(t+l)''},g_{1}^{i}g_{2}^{j})=r_{11}^{i(k+s)'}r_{12}^{j(k+s)'}r_{21}^{i(t+l)''}r_{22}^{j(t+l)''}.$$
Then,
\begin{eqnarray*}
&&\R(g_{1}^{s}g_{2}^{t},g_{1}^{i}g_{2}^{j})\R(g_{1}^{k}g_{2}^{l},g_{1}^{i}g_{2}^{j})\Phi(g_{1}^{i}g_{2}^{j},g_{1}^{s}g_{2}^{t},g_{1}^{k}g_{2}^{l})\\
&&=r_{11}^{i(k+s)}r_{12}^{j(k+s)}r_{21}^{i(t+l)}r_{22}^{j(t+l)}\zeta_{m}^{a[\frac{k+s}{m}]i}
\zeta_{n}^{b[\frac{k+s}{m}]j}\zeta_{n}^{d[\frac{t+l}{n}]j}\\
&&=r_{11}^{i[(k+s)'+[\frac{k+s}{m}]m]}r_{12}^{j[(k+s)'+[\frac{k+s}{m}]m]}r_{21}^{i[(t+l)''+[\frac{t+l}{n}]n]}r_{22}^{j[(t+l)''+[\frac{t+l}{n}]n]}\\
&&\;\;\;\;\times \zeta_{m}^{a[\frac{k+s}{m}]i}
\zeta_{n}^{b[\frac{k+s}{m}]j}\zeta_{n}^{d[\frac{t+l}{n}]j}\\
&&=r_{11}^{i(k+s)'}\zeta_{m}^{-a[\frac{k+s}{m}]i}r_{12}^{j(k+s)'}\zeta_{n}^{-b[\frac{k+s}{m}]j}r_{21}^{i(t+l)''}r_{22}^{j(t+l)''}\zeta_{n}^{-d[\frac{t+l}{n}]j}\\
&&\;\;\;\;\times \zeta_{m}^{a[\frac{k+s}{m}]i}
\zeta_{n}^{b[\frac{k+s}{m}]j}\zeta_{n}^{d[\frac{t+l}{n}]j}\\
&&=r_{11}^{i(k+s)'}r_{12}^{j(k+s)'}r_{21}^{i(t+l)''}r_{22}^{j(t+l)''}\\
&&=\R(g_{1}^{s}g_{2}^{t}\cdot g_{1}^{k}g_{2}^{l},g_{1}^{i}g_{2}^{j}).
\end{eqnarray*}
The sufficiency is proved.
\end{proof}
\begin{remark} \emph{(a) Clearly, if there is a quasi-bicharacter with respect to $\Phi_{a,b,d}$, then either $a=0$ or $m$ is even and $a=\frac{m}{2}$. Similarly,  either $d=0$ or $n$ is even and $d=\frac{n}{2}$.}

\emph{(b) A quasi-bicharacter is said to be \emph{skew-symmetric} if $\R(x,y)\R(y,x)=1$ for all $x,y\in G$. In view of the above proposition, a quasi-bicharacter $\R$ is skew-symmetric if and only if $r_{11}^2=r_{22}^2=1$ and $r_{12}=r_{21}^{-1}.$ }
\end{remark}

\noindent\textbf{Convention.} Given any $\Phi_{a,b,d}$,  let $A_{a,b,d}$ be the set of sequences of $(r_{11},r_{12},r_{21},r_{22})$ satisfying equation (3.14). An element $(r_{11},r_{12},r_{21},r_{22})\in A_{a,b,d}$ is denoted by $\underline{r}$ and the quasi-bicharacter $\R$ determined by $\underline{r}$ is denoted by $\R_{\underline{r}}$. Note that $A_{a,b,d}$ might be an empty set.

\section{Braided monoidal structures on linear Gr-categories over $\Z_m \times \Z_n$ }

Let $G$ be a group. By $\Vec_G$ we denote the category of finite-dimensional $G$-graded spaces. Define the tensor product in this category by the formula $$(V\otimes W)_{g}=\bigoplus_{x,y\in  G, xy=g}V_{x}\otimes W_{y}$$
for all $x,y,g\in G.$ The linear Gr-category over $G$ is $(\Vec_G,\otimes)$. It is well-known that $(\Vec_G,\otimes)$ is  a monoidal category, and the monoidal structures on $\Vec_G$ are parameterized by the normalized 3-cocycles on $G.$ Given a normalized 3-cocycle $\Phi$ on $G,$ we denote the corresponding monoidal category by $\Vec_{G}^{\Phi}.$ Clearly, the monoidal category $\Vec_{G}^{\Phi}$ is braided if and only if there exists a quasi-bicharacter $\R$ with respect to $\Phi.$ The resulted braided monoidal category is written as $(\Vec_{G}^{\Phi},\R).$

Obviously, to classify the monoidal structures on $(\Vec_{G},\otimes)$ is equivalent to determining all the normalized $3$-cocycles, and to classify the braided structures on $\Vec_{G}^{\Phi}$ is equivalent to determining the quasi-bicharacters with respect to $\Phi$. Now we are ready to give our main classification results based on the preparations in the previous sections.

\begin{theorem} Let $G=\mathbb{Z}_{m}\times \mathbb{Z}_{n}$. Then any monoidal structure on the linear Gr-category over $G$ is tensor equivalent to $\Vec_{G}^{\Phi_{a,b,d}}$
for some $a\in \{0,\ldots,m-1\},
b\in \{0,\ldots,(m,n)-1\}$ and $d\in \{0,\ldots,n-1\}$.
\end{theorem}
\begin{proof} Let $\Vec_{G}^{\Phi}$ be a monoidal structure on $(\Vec_{G},\otimes)$. By Proposition 3.9, $\Phi$ is cohomologous to ${\Phi_{a,b,d}}$ for some $a\in \{0,\ldots,m-1\},
b\in \{0,\ldots,(m,n)-1\}$ and $d\in \{0,\ldots,n-1\}$. At the same time, it is well-known that any two monoidal categories $\Vec_{G}^{\Phi_{1}}$,  $\Vec_{G}^{\Phi_{2}}$ are tensor equivalent provided $\Phi_{1}$  and $\Phi_{2}$ are cohomologous.
\end{proof}

\begin{theorem}Let $G=\mathbb{Z}_{m}\times \mathbb{Z}_{n}$. Then any  braided monoidal structure on $(\Vec_{G},\otimes)$   is tensor equivalent to $(\Vec_{G}^{\Phi_{a,b,d}},\R_{\underline{r}})$
for some $a\in \{0,\ldots,m-1\},
b\in \{0,\ldots,(m,n)-1\}$, $d\in \{0,\ldots,n-1\}$ and $\underline{r}\in A_{a,b,d}$.
\end{theorem}
\begin{proof} Let $(\Vec_{G}^{\Phi},\R)$ be a  braided monoidal structure on $(\Vec_{G},\otimes)$. Similarly, $\Phi$ is cohomologous to ${\Phi_{a,b,d}}$ for some $a\in \{0,\ldots,m-1\},
b\in \{0,\ldots,(m,n)-1\}$ and $d\in \{0,\ldots,n-1\}$ by Proposition 3.9. So ${\Phi_{a,b,d}}=\Phi\delta(f)$ for a $2$-cochain $f\in \Hom_{\mathbb{Z}G}({B}_{2},k^{\ast})$. Define $\R'$ by $$\R'(x,y):=\R(x,y)f(x,y)f(y,x)^{-1},\;\;\;\;\;x,y\in G.$$
Direct computations show that $\R'$ is a quasi-bicharacter with respect to $\Phi_{a,b,d}$. Thus $\R'=\R_{\underline{r}}$ for some
$\underline{r}\in A_{a,b,d}$ by Proposition 3.11. Therefore  $(\Vec_{G}^{\Phi},\R)$ is tensor equivalent to $(\Vec_{G}^{\Phi_{a,b,d}},\R_{\underline{r}})$.
\end{proof}

\begin{remark} \emph{It is not hard to see that the braided monoidal category $(\Vec_{G}^{\Phi},\R)$ is symmetric if and only if $\R$ is skew-symmetric as defined in Remark 3.12.}
\end{remark}

\end{document}